\documentclass[english,a4paper,12pt]{amsart}
\usepackage{graphicx}
\usepackage[utf8]{inputenc}
\usepackage[T1]{fontenc}
\usepackage{amsmath,amssymb, amsthm}
\usepackage{babel}
\usepackage{nccmath} 
\usepackage{color}
\usepackage{hyperref}

\usepackage[left=2.4cm, top=2.8cm, bottom=2.8cm, right=2.4cm]{geometry}
\usepackage{bm}

\usepackage{mathtools}
\newcommand{\defeq}{\vcentcolon=}

\makeatletter
\renewcommand{\paragraph}{\@startsection%
{paragraph}
{4}
{0pt}
{\bigskipamount}
{-0.5em}
{\normalfont\normalsize\bfseries}}
\makeatother

\newtheoremstyle{thm}
{6pt plus 1pt minus 1pt}
{6pt plus 1pt minus 1pt}
{\slshape}
{}
{\scshape}
{.}
{.5em}
{}

\newtheoremstyle{thm*}
{6pt plus 1pt minus 1pt}
{6pt plus 1pt minus 1pt}
{\slshape}
{}
{\scshape}
{.}
{.5em}
{\thmname{#1}\thmnote{ #3}}

\newtheoremstyle{def}
{6pt plus 1pt minus 1pt}
{6pt plus 1pt minus 1pt}
{}
{}
{\scshape}
{.}
{.5em}
{}

\theoremstyle{thm}
\newtheorem{theorem}{Theorem}[section]
\newtheorem{proposition}[theorem]{Proposition}
\newtheorem{lemma}[theorem]{Lemma}

\theoremstyle{thm*}
\newtheorem*{theorem*}{Theorem} 

\theoremstyle{def}
\newtheorem{example}[theorem]{Example}
\newtheorem{definition}[theorem]{Definition}
\newtheorem{remark}[theorem]{Remark}

\newcommand{\R}{\mathbb{R}}
\newcommand{\C}{\mathbb{C}}
\newcommand{\Z}{\mathbb{Z}}

\DeclareMathOperator{\T}{\mathit{T}}
\newcommand{\Int}{\mathop{\mathrm{Int}}\nolimits}
\newcommand{\Bd}{\partial}
\renewcommand{\:}{\,{:}\;}
\newcommand{\hT}{\mathop{{}^{h}T}\nolimits}

\newcommand{\id}{\mathop{\mathrm{id}}\nolimits}

\begin{document}
\title[Pseudoconcave boundary of a complex surface]{On the strongly pseudoconcave boundary\\[2pt]
of a compact complex surface}


\author[N. Kasuya]{Naohiko Kasuya}

\address{Department of Mathematics, Faculty of Science, Kyoto Sangyo University,  
Kamigamo Motoyama, Kita-ku, Kyoto 603-8555 Japan.}

\email{nkasuya@cc.kyoto-su.ac.jp}

\author[D. Zuddas]{Daniele Zuddas}

\address{Dipartimento di Matematica e Geoscienze, Università di Trieste, Via Valerio 12/1, 34127 Trieste, Italy.}

\email{dzuddas@units.it}

\begin{abstract}
In this paper, we establish the method of holomorphic handle attaching to the strongly pseudoconcave boundary of a complex surface. 
We use this for proving the following statements: 
(1) every closed connected oriented contact $3$-manifold can be filled as the strongly pseudoconcave boundary of a compact complex surface; 
(2) any two closed connected oriented contact $3$-manifolds are complex cobordant. Moreover, we show that such complex surface (or complex cobordism) can be taken Kähler.
\end{abstract}

\maketitle

\section{Introduction}
Let $(M, \xi)$ be a closed connected positively oriented contact 3-manifold, and suppose that the contact structure $\xi$ is the kernel of a global contact 1-form $\alpha$. An important question in contact topology is whether $(M, \xi)$ is \textsl{holomorphically fillable}, namely the problem of understanding whether $(M, \xi)$ can be realized as the strongly \textsl{pseudoconvex} oriented boundary of a compact complex surface $V$, so that $\xi$ coincides with the complex tangency distribution, namely the maximal $J$-invariant distribution in $\Bd V$, where $J$ denotes the complex structure of $V$. By a result of Bogomolov and De Oliveira \cite{BO97}, a holomorphic filling can be always modified into a Stein filling of the same contact 3-manifold. On the other hand, the existence of a Stein filling implies that the given contact 3-manifold $(M, \xi)$ is tight. Then, there is an obstruction for the existence of a holomorphic filling.

A related question is whether $(M, \xi)$ can be filled as the strongly \textsl{pseudoconcave} boundary of a compact complex surface, and try to understand if this kind of filling is obstructed as well. In this paper we will show that there are no obstructions, namely a concave filling always exists.

We recall that an oriented real hypersurface $\Sigma$ in a complex $n$-manifold $W$, $n \geq 2$, is said to be \textsl{strongly pseudoconvex} if there is a smooth regular strictly plurisubharmonic function $\phi \: U \to \R$ defined in a neighborhood $U$ of $\Sigma$ in $W$, such that $\Sigma = \phi^{-1}(0)$ and $\Sigma$ is oriented as the boundary of the sublevel of $\phi$. The oriented hypersurface $\Sigma$ is said to be \textsl{strongly pseudoconcave} if it becomes strongly pseudoconvex by reversing its orientation.

Strong pseudoconvexity (resp. pseudoconcavity) of $\Sigma $ is equivalent to strong pseudoconvexity (resp. pseudoconcavity) of the CR structure $\T^{(1,0)}\Sigma$. 
This in turn is equivalent to the fact that the holomorphic tangent space of $\Sigma$, namely $\hT \Sigma \defeq \T\Sigma \cap J(\T \Sigma) $, is a contact structure on $\Sigma$ satisfying a certain  condition on the orientations that discriminates pseudoconvexity from pseudoconcavity, as one can easily see by considering $\hT \Sigma$ as the kernel of the 1-form $(d\phi \circ J)_{|\T\Sigma}$.

Let $V$ be a compact complex $n$-manifold with smooth boundary $\Bd V$, 
and let $M$ be one of its boundary components. 
Then, by definition, $V$ is contained in a larger complex $n$-manifold $V'$ without boundary, and so $M$ is a smooth real hypersurface in $V'$. 
Thus, according with the above definitions, it makes sense of speaking about strong pseudoconvexity or pseudocancavity of $M$ as a boundary component of $V$, since this does not depend on how $V$ is holomorphically embedded in $V'$.

\begin{definition}
A contact manifold $(M,\xi)$ is said to be \textsl{oriented} if both the underlying manifold $M$ and the contact distribution $\xi$ are oriented. In this case, $\xi$ is cooriented and so it is the kernel of a global contact form.
\end{definition}
Notice that if $V$ is a complex manifold with strongly pseudoconvex or pseudoconcave boundary, then $(\Bd V, \hT \Bd V)$ has a preferred orientation as contact manifold: $V$ is oriented as the boundary of $V$, and the contact distribution $\hT \Bd V$ on $\Bd V$ is oriented by the complex structure of $V$.

\begin{definition}
A \textsl{convex holomorphic filling} of a closed oriented contact manifold $(M, \xi)$ 
is a compact complex manifold $V$ such that $(\Bd V, \hT \Bd V)$ is strongly pseudoconvex and contactomorphic to $(M, \xi)$. 
Similarly, $V$ is a \textsl{concave holomorphic filling} of $(M, \xi )$ if $(\Bd V, \hT \Bd V)$ is  strongly pseudoconcave and contactomorphic to $(-M, \xi)$, where $-M$ denotes $M$ with the reversed orientation. 
\end{definition}

\begin{definition}
A \textsl{convex complex cobordism} from an oriented contact manifold $(M_1, \xi _1)$ to an oriented contact manifold $(M_2, \xi _2)$ 
is a compact connected complex manifold $V$ such that the followings hold:
\begin{enumerate}
\item $\Bd V$ is the disjoint union of two closed not necessarily connected manifolds $N_1$ and $N_2$;
\item $N_1$ is strongly pseudoconcave and $N_2$ is strongly pseudoconvex;
\item $(N_1, \hT N_1)$ is contactomorphic to $(-M_1, \xi_1)$ and $(N_2,\hT N_2)$ is contactomorphic to $(M_2, \xi_2)$.
\end{enumerate}
In this case, $N_1$ is called the concave boundary and $N_2$ is called the convex boundary of $V$.
We say also that $V$ is a \textsl{concave complex cobordism} from $(-M_2, \xi _2)$ to $(-M_1, \xi _1)$. 
\end{definition}

If $(M,\xi)$ is a contact manifold such that $\dim M \equiv 3 \pmod4$, then $\xi$ determines a preferred orientation on $M$, which is induced by the volume form $\alpha \wedge (d \alpha)^{2k+1}$, where $\alpha$ is any contact form defining $\xi$, and $\dim M = 4k+3$. If in addition $(M,\xi)$ is endowed with any orientation as a contact manifold, then we say that $\xi$ is a \textsl{positive contact structure}, and that $(M,\xi)$ is a \textsl{positive contact manifold}, if the given orientation of $M$ agrees with the preferred one. For a complex manifold $V$ of even complex dimension, a boundary component $N \subset \Bd V$ is strongly pseudoconvex (resp. strongly pseudoconcave) if and only if the contact structure $\hT N$ is positive (resp. negative).

Now, we consider contact 3-manifolds. Firstly, note that convex or concave holomorphic fillings are only defined for positive contact 3-manifolds. However, convex complex cobordisms are only defined between positive contact 3-manifolds, while the concave ones are defined between negative contact 3-manifolds. This discrepancy is due to the fact that we are mostly interested in positive contact 3-manifolds, so for them we define holomorphic fillings of the above types, while for cobordisms we prefer to keep the usual orientation conventions.

\begin{remark}
Eliashberg \cite{El85} proved that any positive contact $3$-manifold is complex cobordant to itself. 
However, two strongly pseudoconvex CR $3$-manifolds with the same underlying contact structure are not necessarily complex cobordant (see \cite{Ro65}). 
Hence, a complex cobordism from $(M_1, \xi _1)$ to $(M_2, \xi _2)$ and a complex cobordism from $(M_2, \xi _2)$ to $(M_3, \xi _3)$ cannot necessarily be glued to give a complex cobordism from $(M_1, \xi _1)$ to $(M_3, \xi _3)$. So, at this point is not clear whether complex cobordism is a transitive relation between closed connected positive contact 3-manifolds.  Theorem \ref{main} below demonstrates that this relation is actually trivial.
\end{remark}

We are now ready to state our main theorems.

\begin{theorem}\label{main2}
Any closed connected positive contact $3$-manifold admits infinitely many pairwise inequivalent concave holomorphic fillings. Moreover, both Kähler and non-Kähler concave holomorphic fillings exist for any contact 3-manifold.
\end{theorem}

\begin{theorem}\label{main}
Let $(M_1, \xi _1)$ and $(M_2, \xi _2)$ be any two closed conneted positive contact $3$-manifolds. 
Then, there exists a Kähler convex complex cobordism from $(M_1, \xi _1)$ to $(M_2, \xi _2)$.
\end{theorem}

Our approach is based on a variation of Eliashberg's construction of Stein manifolds, 
and moreover we establish the method of holomorphic handle attaching 
to the strongly pseudoconcave boundary of a complex surface.

In \cite{El90}, Eliashberg described the handlebody construction of Stein manifolds. 
Roughly speaking, his construction consists of two parts. 
First, he took the standard holomorphic $k$-handle of complex dimension $n$, whose core is a $k$-dimensional totally real disk, 
and showed that if $k \leq n$ this handle can be attached analytically along an \textsl{isotropic} $(k-1)$-sphere embedded in the convex boundary of a given complex $n$-manifold endowed with a certain strictly plurisubharmonic function.
Next, he extended the original strictly plurisubharmonic function over the handle.

In partial contrast, our proof of Theorem \ref{main2} goes as follows (for $k= n=2$). 
First, we take the standard holomorphic handle of index $2$ whose core is a \textsl{holomorphic} disk, 
and attach it along a transverse knot in the strongly pseudoconcave boundary of a given complex surface endowed with a certain strictly plurisuperharmonic function (Proposition \ref{attach}).
Next, we modify the original function to obtain a new strictly plurisuperharmonic function 
whose level set defines the new concave boundary (Proposition \ref{concave boundary}). 
The point is that we do not have to extend the function over the whole handle, 
because we need a function only near the boundary.

If the attaching circle is the positive push-off of a Legendrian knot $K$,  
then the effect on the contact structure on the concave boundary is a contact surgery along $K$. 
Since the attaching circle is a transverse knot, the framing can be chosen arbitrarily. 
In particular, both contact $(\pm 1)$-surgeries can be realized (Theorem \ref{cobordism}). 
This means that holomorphic handle attaching to the concave boundary of a complex surface is unrestricted, and then we are able to obtain all closed positive contact 3-manifolds. Moreover, we show that this handle operation preserves the property of being a Kähler complex surface.

Then, we are left to produce a compact complex surface with strongly pseudoconcave boundary to start with. This may be done by removing the interior of a strongly pseudoconvex $4$-ball from a compact complex surface. Then, the complement is a concave holomorphic filling of the standard contact $3$-sphere.

In \cite{DKZ17}, Antonio J. Di Scala and we constructed the first examples of non-K\"{a}hler complex structures on $\R^4$.
We denote such a complex surface by $E$.
It is non-K\"{a}hler because it is diffeomorphic to $\R^4$ and contains compact elliptic curves (see also \cite{DKZ18} for some further properties of these surfaces).
We recall the notion of Calabi-Eckmann type complex manifold introduced in \cite{DKZ17}. 

\begin{definition}
A complex manifold $Y$ is said to be of Calabi-Eckmann type if there exist a closed complex manifold $X$ of positive dimension, and a holomorphic immersion $k \: X \to Y$ which is null-homotopic as a continuous map. 
\end{definition}

Notice that a Calabi-Eckmann type complex manifold is non-K\"{a}hler by Stokes' theorem. \pagebreak
The next theorem has been proven in \cite{KZ18} by taking a certain 4-ball embedded in the complex surface $E$, with the induced complex structure. 

\begin{theorem}[\cite{KZ18}]\label{4-ball}
The 4-ball $B^4$ admits a complex structure of Calabi-Eckmann type with strongly pseudoconcave boundary. 
Moreover, the negative contact structure on $\Bd B^4=S^3$ is overtwisted and homotopic as a plane field 
to the positive standard contact structure on $S^3$. 
\end{theorem}

Applying the holomorphic handle attaching method (Theorem \ref{cobordism}) to this example, we actually obtain the following result.

\begin{theorem}\label{CE}
Every closed conneted positive contact $3$-manifold admits a concave holomorphic filling of Calabi-Eckmann type.
\end{theorem}

The organization of this paper is as follows. 
In Section \ref{def}, we recall some basic definitions and results that will be useful through the paper. 
Moreover, we summarize the theory of contact Dehn surgery, 
and prepare the standard model of a concave holomorphic handle which will be used in our construction. 
In Section \ref{HHA}, we establish 
the method of holomorphic handle attaching to the concave boundary of a complex surface and we use this to prove Theorem \ref{cobordism}. 
This is the main part of our construction. 
Finally, in Section \ref{proofs}, we prove Theorems \ref{main2} and \ref{main} by using Theorem \ref{cobordism}.

\section{Preliminaries}\label{def}

\paragraph{Pseudoconvexity and pseudoconcavity}
Let $\phi \colon W\to \R$ be a smooth function on a complex manifold $W$ of complex dimension $n$. We recall that the \textsl{complex Hessian} of $\phi$ is the Hermitian form on the complex tangent bundle $\T W \otimes \C$ defined by \[H_\phi(v,w) \defeq i(\partial\bar\partial \phi)(v, J w),\] where $J$ denotes the complex structure of $W$ and $v,w \in T_p\, W \otimes \C$. In local holomorphic coordinates, $H_\phi$ can be expressed by the usual formula \[H_\phi = \sum_{\alpha, \beta = 1}^n \frac{\partial^2 \phi}{\partial z_\alpha\, \partial\bar z_\beta}\;  dz_\alpha \otimes d\bar z_\beta.\]

The function $\phi$ is said to be \textsl{strictly plurisubharmonic} 
if its complex Hessian $H_\phi$ is positive definite at every point of $W$.
If instead $H_\phi$ is negative definite at every point of $W$, namely if $-\phi$ is strictly plurisubharmonic, then $\phi$ is said to be \textsl{strictly plurisuperharmonic.}

Let $M \subset W$ be a smooth oriented real hypersurface in a complex manifold $W$. Suppose that $M$ is the zero level set of a smooth regular function $\phi \: U \to \R$, that is $M = \phi^{-1}(0)$, where $U$ is an open neighborhood of $M$ in $W$. We assume that $M$ is oriented as the boundary of the sublevel $\phi^{-1}(\mathopen]-\infty, 0\mathclose])$. 
The \textsl{Levi form} $\mathcal L_{M,\phi}$ of $M$ with respect to $\phi$ is defined as the restriction of $H_\phi$ to the subbundle $\hT M \otimes \C \subset (\!\T W \otimes \C)_{|M}$, namely $\mathcal L_{M,\phi} = H_{\phi \mid \hT\! M \otimes \C}$. 
If $\psi$ is another defining function for $M$ as above, then the corresponding Levi form $\mathcal L_{M,\psi}$ satisfies $\mathcal L_{M,\psi} = \lambda \mathcal L_{M,\phi}$ for some smooth function $\lambda \: M \to \mathopen]0,+\infty\mathclose[$. So, the Levi form can be considered up to a smooth positive coefficient, and in this sense it depends only on the hypersurface $M \subset W$, and shall be indicated by $\mathcal L_M$. Observe that $\mathcal L_{-M} = -\mathcal L_M$.

The oriented hypersurface $M$ is said to be \textsl{strongly pseudoconvex} if $\mathcal L_M$ is positive definite at every point of $M$; if instead $\mathcal L_M$ is negative definite at every point of $M$, then $M$ is said to be \textsl{strongly pseudoconcave}.

It is well-known that for a complex manifold $W$, the following three conditions are equivalent:
\begin{enumerate}
\item
$W$ is Stein;
\item
$W$ admits a proper holomorphic embedding into $\C^N$ for some $N > \dim_\C W$;
\item
$W$ admits a smooth strictly plurisubharmonic proper function. 
\end{enumerate}
This is a classical result by Remmert \cite{Re56}, Bishop \cite{Bi61}, Narashiman \cite{Na60}, and Grauert\cite{Gr58}. 
Moreover, a complete topological characterization of Stein manifolds has been done by Eliashberg \cite{El90} for $n>2$, 
and by Eliashberg \cite{El90} and Gompf \cite{Go98} for $n=2$ (see also Loi-Piergallini \cite{LP01}). 

\paragraph{Fillings and cobordisms} A vector field $v$ on a symplectic manifold $(X, \omega )$ is called a Liouville vector field if $L_{v}\omega =\omega $. 
Suppose that $X$ is compact with boundary. 
The boundary $\Bd X$ is called $\omega $-convex (resp. $\omega $-concave) 
if there exists a Liouville vector field $v$ near $\Bd X$, pointing outwards (resp. inwards) along $\Bd X$. 
In this case, $\xi =\ker (i_{v}\omega_{|\Bd X})$ defines 
a positive (resp. negative) contact structure on the boundary $\Bd X$. 

\begin{definition}\label{fillings/def}
Let $(M, \xi )$ be a closed connected oriented contact $(2n-1)$-manifold. 
\begin{enumerate}
\item
A Stein filling of $(M, \xi )$ is a compact Stein manifold 
whose boundary is strongly pseudoconvex and contactomorphic to $(M, \xi )$. 
In this case we say that $(M, \xi )$ is Stein fillable. 
\item
A strong symplectic filling (or a convex symplectic filling) of $(M, \xi )$ 
is a compact symplectic manifold $(X, \omega )$ 
whose boundary is $\omega $-convex and contactomorphic to $(M, \xi )$. 
In this case we say that $(M, \xi )$ is strongly symplectically fillable. 
\item
A weak symplectic filling of $(M, \xi )$ is a compact symplectic manifold $(X, \omega )$ 
such that $\Bd X=M$ as oriented manifolds and $(\omega_{|\xi }) ^{n-1} >0$. 
In this case we say that $(M, \xi )$ is weakly symplectically fillable. 
\item
A concave symplectic filling of $(M, \xi )$ is a compact symplectic manifold $(X, \omega )$ 
whose boundary is $\omega $-concave and contactomorphic to $(-M, \xi )$. 
\end{enumerate}
\end{definition}

\begin{definition}
Let $(M_1, \xi_1)$ and $(M_2, \xi _2)$ be closed connected oriented contact $(2n-1)$-manifolds. 
\begin{enumerate}
\item
A Stein cobordism from $(M_1, \xi_1)$ to $(M_2, \xi _2)$ is a complex cobordism $V$ from $(M_1, \xi_1)$ to $(M_2, \xi _2)$  
admitting a strictly plurisubharmonic function $\phi $ such that $M_1$ and $M_2$ are non-singular level sets of $\phi $, and satisfying $\phi(M_1) < \phi(x) < \phi(M_2)$ for all $x \in \Int V$.
In this case we say that $(M_1, \xi _1)$ is Stein cobordant to $(M_2, \xi _2)$. 
\item
A symplectic cobordism from $(M_1, \xi_1)$ to $(M_2, \xi _2)$ is 
a compact symplectic $2n$-manifold $X$ such that the boundary $\Bd X$ consists of a 
$\omega $-convex part $\Bd _{+}X$ and a $\omega $-concave part $\Bd _{-}X$ 
which are contactomorphic to $(M_2, \xi _2)$ and $(-M_1, \xi _1)$, respectively. 
In this case we say that  $(M_1, \xi_1)$ is symplectically cobordant to $(M_2, \xi _2)$. 
\end{enumerate}
\end{definition}

These fillings and cobordisms have been widely studied. 
There are several important known facts. First, the following implications are immediate for a given contact $m$-manifold:
\[\text{Stein fillable} \Rightarrow \text{strongly symplectically fillable} \Rightarrow \text{weakly symplectically fillable.}\]

It is known that these implications are not equivalences when $m=3$. 
Moreover, it has been shown that 
a weakly symplectically fillable contact $m$-manifold is tight, namely, not overtwisted, for all odd $m\geq 3$. 
For the definition of overtwistedness for $m\geq 5$, see \cite{BEM}. 
On the other hand, Etnyre and Honda \cite{EH02} have proved the following theorem. 

\begin{theorem}[Etnyre-Honda \cite{EH02}]\label{EH}
Any closed positive contact $3$-manifold admits infinitely many concave symplectic fillings.
\end{theorem}

This shows that for contact $3$-manifolds, concave symplectic fillings are not restrictive at all. 
In fact, they first showed the following. 

\begin{theorem}[Etnyre-Honda \cite{EH02}]\label{EH2}
Any closed positive contact $3$-manifold is Stein cobordant to a Stein fillable contact $3$-manifold. 
\end{theorem}

Then, Theorem \ref{EH} follows from Theorem \ref{EH2} 
and the following result by Lisca and Mati\'{c} \cite{LM97}. 

\begin{theorem}[Lisca-Mati\'{c} \cite{LM97}]\label{LM}
Every Stein filling of a contact manifold is biholomorphic to a domain in a smooth complex projective manifold. Moreover, the biholomorphism can be chosen to be a symplectomorphism with respect to the Stein symplectic structure and a Kähler form on the projective manifold. 
\end{theorem}

By Theorem \ref{EH2}, any contact $3$-manifold $(M, \xi )$ admits 
a Stein cobordism $V$ to a Stein fillable contact $3$-manifold $(N, \eta )$. 
Let $W$ be a Stein filling of the contact manifold $(N, \eta )$. 
By Theorem \ref{LM}, $W$ is biholomorphic and symplectomorphic to a domain $X$ in a complex projective surface $S$. 
Then, $S- \Int X$ is a concave symplectic filling of $(N, \eta )$, 
and the gluing of $V$ and $S-\Int X$ yields a concave symplectic filling of $(M, \xi )$. 
Notice that when considering symplectic cobordisms, 
the gluing of boundary components can be done if they are contactomorphic. 

On the other hand, there are also some results on Stein cobordisms and complex cobordisms from CR geometry.

\begin{theorem}[Epstein-Henkin \cite{EH01}]
If a CR $3$-manifold is Stein cobordant to a fillable CR $3$-manifold, then it is also fillable. 
\end{theorem}

\begin{theorem}[De Oliveira \cite{De03}]\label{De Oliveira}
There exists a complex cobordism such that 
the convex boundary is a fillable CR $3$-manifold, and the concave boundary is a non-fillable CR $3$-manifold. 
In particular, there exists a compact complex surface 
whose boundary is a connected strongly pseudoconvex non-fillable CR 3-manifold with reversed orientation. 
\end{theorem}

We note that Theorems \ref{main2} and \ref{main} imply this theorem of De Oliveira (Theorem 1 in \cite{De03}), while our construction is quite different from his. 
In fact, his proof  is done by plumbing  strongly pseudoconcave neighborhoods of compact curves each of which is positively embedded in a projective surface. The non-fillability holds at the CR structure level in his example, and he does not discuss holomorphic fillability at the contact structure level. In this sense, our results are stronger than Theorem 1 in \cite{De03}, however, his Theorems 2 and 3 are different important results about complex cobordism and fillability of CR $3$-manifolds, which do not follow from our results. 

Theorems \ref{main2} and \ref{main} can be also seen as the complex version of Etnyre-Honda's results.

\paragraph{Contact surgery}
For a knot $K$ in any $3$-manifold $M$,  
let $M_{p/q}(K)$ denote the $3$-manifold obtained from $M$ by $(p/q)$-Dehn surgery along $K$ with respect to a given reference framing. 

Now, suppose $K$ is a Legendrian knot in a contact $3$-manifold $(M, \xi )$.  
By the Legendrian neighborhood theorem, there exists a tubular neighborhood $\nu K$ of $K$ 
contactomorphic to the contact solid torus $(N_{\delta }, \eta )$ with convex boundary, 
where 
\begin{equation}\label{Ndelta/eqn}
N_{\delta }=\left\{(\theta, x, y) \in S^1\times \R^2 \mid x^2+y^2\leq \delta ^2 \right\}\subset S^1\times \R^2
\end{equation}
and \[\eta = \ker(\cos \theta\, dx - \sin \theta\, dy),\] 
where $\theta$ is the angular coordinate on $S^1$ and $(x,y)$ are the Cartesian coordinates on $\R^2$.

On the boundary torus $T_{\delta }=\Bd N_{\delta }$, 
let $\mu $ be the meridian and $\lambda $ be the longitude determined by the contact framing. 
Then, the dividing set of the convex torus $T_{\delta }$ 
consists of two parallel copies of the longitude $\lambda $. 
By Giroux's flexibility \cite{Gi91}, 
the dividing set of a convex surface determines the contact structure on a neighborhood of the surface. 
Hence, two contact $3$-manifolds can be glued along convex surfaces if they have the same dividing sets. 

For any integer $k$, a contact structure on the manifold $M_{1/k}(K)$ can be constructed as follows. 
The manifold $M_{1/k}(K)$ is obtained from $M$ by removing $\nu K$ and gluing in $S^1\times D^2$ 
so that the meridian is sent to $\mu +k \lambda $ on $\Bd (M-\nu K)$. 
By Honda's classification of tight contact structures on the solid torus \cite{Ho00}, 
there is a unique tight contact structure on $S^1\times D^2$ that extends the contact structure $\xi_{| M-\nu K}$. 
Thus, we obtain a contact structure $\xi '$ on the manifold $M_{1/k}(K)$. 
In this case, we say that $(M_{1/k}(K), \xi')$ is obtained from $(M, \xi )$ by contact $(1/k)$-surgery on $K$. 
In particular, contact $(\pm 1)$-surgery on a Legendrian knot is well-defined. 
Ding and Geiges \cite{DG} have proved the following interesting result on contact $(\pm1)$-surgery.

\begin{theorem}[Ding-Geiges \cite{DG}]\label{DG}
Any closed conneted positive contact $3$-manifold can be obtained from 
the standard contact $3$-sphere by a sequence of contact $(\pm 1)$-surgeries. 
\end{theorem}

In the following, $(M_K^{\pm}, \xi _K^{\pm})$ 
denotes the result of contact $(\pm 1)$-surgery on $K$. 

\paragraph{The standard concave handle}
The following fact is classical and well-known, but since it plays an important role in our paper, we outline a proof of it for the reader convenience.

\begin{proposition}\label{classical}
Let $U \subset \C^n$ be a non-empty open subset, and let $\phi \: U \to \R$ be a smooth function.
Let $\Omega \subset \C^{n+1} \cong \C^n \times \C$ be the submanifold with boundary defined by 
\[\Omega = \{(\bm z, z_{n+1}) \in U \times \C \mid |z_{n+1}|^2 \leq \exp (\phi (\bm z))\}.\]
Then, \[\Bd \Omega = \{(\bm z, z_{n+1}) \in \Omega \mid |z_{n+1}|^2 = \exp (\phi (\bm z))\}\] is strongly pseudoconcave (resp. strongly pseudoconvex) in $\C^{n+1}$
if and only if $\phi$ (resp. $-\phi)$ is a strictly plurisubharmonic function on $U$. 
\end{proposition}
\begin{proof}
Consider the function $g \: \Omega \to \R$ defined by $g(\bm z, z_{n+1}) = |z_{n+1}|^2 - \exp(\phi(\bm z))$, with $\bm z = (z_1, \dots, z_n)$. A direct computation in coordinates shows that \[H_g = -e^\phi (H_\phi + \partial \phi \otimes \bar\partial \phi) + dz_{n+1} \otimes d\bar z_{n+1}.\]

Let $p=(\bm z,z_{n+1}) \in \Bd \Omega$, and let $v \in \T_{\! p} \C^{n+1} \otimes \C$ be a $(1,0)$-vector. Then $v \in \hT_{\! p} \Bd \Omega \otimes \C$ if and only if $v$ is in the kernel of \[\partial g = -e^\phi\, \partial \phi + \bar z_{n+1}\, dz_{n+1}.\] By setting \[v = w + v_{n+1}\, \frac{\partial}{\partial z_{n+1}}, \] with \[w = \sum_{k=1}^n v_k\, \frac{\partial}{\partial z_k}\] and $v_k \in \C$ for all $k$, the vector $v \in \hT_{\! p} \Bd \Omega \otimes \C$ is uniquely determined by $w$, which may be arbitrarily chosen in $T^{(1,0)}_{\bm z}\, U$, because $z_{n+1} \neq 0$ on $\Bd \Omega$. 
This and the fact that $|z_{n+1}|^2 = \exp(\phi(\bm z))$ on $\Bd \Omega$, imply that \[\mathcal L_{\Bd \Omega}(v, \bar v) = H_g(v, \bar v) = -\exp(\phi(\bm z))\, H_\phi(w, \bar w). \qedhere\]
\end{proof}

\begin{example}
Let \[\Omega _a= \left\{(z_1, z_2) \in \C^2 \mid |z_2|\leq \exp\! \left(\dfrac{|z_1|^2}{a}-a\right) \right\}\] for $a>0$. 
Then, $\Bd\Omega _a$ is a strongly pseudoconcave 3-manifold in $\C^2$.  
\end{example}

We put
$E=\{(z_1, z_2) \in \C^2 \mid z_2=0\}$. 

\begin{proposition}\label{psi/thm}
On the domain $\Omega_a - E$, there exists a strictly plurisuperharmonic function $\psi_a$ with no critical points 
satisfying $\psi_a ^{-1}(0)=\Bd \Omega _a$ and $\psi_a^{-1}(-\infty , 0]=\Omega _a - E$. 
\end{proposition}
\begin{proof}
Put \[S_c=\left\{(z_1, z_2) \in \C^2 \mid |z_2|=\exp\! \left(\dfrac{|z_1|^2}{c}-c\right) \right\} = \Bd \Omega_c\] for any $c>0$.
Then, $S_c$ is a strongly pseudoconcave hypersurface in $\C^2$ by Proposition \ref{classical}. 
The domain $\Omega _a - E$ is foliated by 
the family of strongly pseudoconcave hypersurfaces $\{ S_c \mid c\geq a \}$. 
Hence, there exists a strictly plurisuperharmonic function $\psi $ 
whose level sets coincide with the hypersurfaces $S_c$. 
It is obviously possible to reparametrize $\psi$ to get a function that shall be denoted by $\psi_a$ such that 
$\psi_a ^{-1}(0)=\Bd \Omega _a$ and $\psi_a^{-1}(-\infty , 0]=\Omega _a - E$. 
\end{proof}

Now, we fix the function $\psi_a$ and set  
\begin{gather*}
H_a=\Omega _{a} \cap \left\{(z_1, z_2) \in \C^2 \mid |z_1|\leq 1+a^{-1} \right\}, \\
H_a(u)=H_a - \psi_a^{-1}(-u, 0], \\
H_a(u, v)=H_a \cap \psi_a ^{-1}[-v, -u] 
\end{gather*}
for positive numbers $u$ and $v$. 
We call $H_a$ the standard concave handle. 
There is an identification of $H_a$ with the standard topological 2-handle $B^2 \times B^2$ given by the diffeomorphism
\begin{equation}\label{h-diffeo/eqn}
\begin{gathered}
h_a \: B^2 \times B^2 \longrightarrow H_a\\
h_a(w_1,w_2) = \left((1+a^{-1}) w_1,\; \exp\!\left(\frac{|w_1|^2}{a} - a \right) w_2 \right).
\end{gathered}
\end{equation}

We make use of the following terminology, which is a bit different from the standard one of handle theory, being adapted to the holomorphic case: the \textsl{core} of $H_a$ is the holomorphic disk $\{z_2 =0\} \cap H_a$, and the circle $|z_1|=1$ in the core is called the \textsl{attaching circle} of $H_a$; the \textsl{attaching region} of $H_a$ is the subset of $H_a$ defined by $1\leq|z_1|\leq1+a^{-1}$. 

Roughly speaking, our idea is to attach $H_a$ 
to the concave boundary of a complex surface $W$ 
along a transverse knot in the contact boundary $\Bd W$. 
Observe that the size of the handle $H_a$ decreases as $a$ increases, and we can take an arbitrarily thin standard handle. 
This is an important point of our handle attaching method and a crucial difference with the case of Weinstein handles.

\section{Holomorphic handle attaching}\label{HHA}
The purpose of this section is to prove the following theorem.

\begin{theorem}\label{cobordism}
Let $W$ be a complex surface with strongly pseudoconcave contact boundary $(M, \hT M)$. 
Consider a Legendrian knot $K$ in $(M, \hT M)$, 
and let $(N, \eta)$ be the result of a $(\pm1)$-contact surgery on $(M, \hT M)$ along $K$. 
Then, there exists a closed collar neighborhood $U$ of  the boundary $M$ and 
a concave holomorphic filling $Y$ of $(-N, \eta )$ such that  $W - \Int U\subset Y$. 
In particular, there exists a concave complex cobordism from $(M, \hT M)$ to $(N, \eta )$. 
\end{theorem}

This will be proved by the holomorphic handle attaching method. 
In contrast to the case of Eliashberg's construction, 
the core of a $2$-handle is a holomorphic disk and the attaching circle is a transverse knot in our case. 
We will make use of the notation 
$\Delta(r)=\left\{z\in \C \mid |z|<r \right\}$ for the open 2-disk of radius $r$. 
Let $(z_1, z_2)$ be the canonical coordinates on $H_a \subset \C^2$ and 
let $D$ denote the core disk $\{z_2=0\}\cap H_a$.

\begin{proposition}\label{attach}
Let $W$ be a complex surface with strongly pseudoconcave boundary $\Bd W$, 
let $\xi $ be the induced contact structure on the boundary, and let $L$ be a transverse knot in $(\Bd W, \xi )$.
Then, the standard concave handle $H_a$ can be holomorphically attached along $L$ with every given framing, for any sufficiently large $a>0$.
\end{proposition}
\begin{proof}
We first consider the special case where $\Bd W$ is real analytic in $W$ in a neighborhood of $L$.
There exists a tubular neighborhood $N$ of $L$ in the contact manifold $(\Bd W, \xi )$ which is real analytic in $W$ and (smoothly)
contactomorphic to the standard model $(N_\delta, \xi_0)$ endowed with the positive contact structure $\xi_0 = \ker(d\theta + r^2d\varphi)$,
so that $L$ corresponds to the circle $S^1 \times \{0\}$ of the solid torus $N_\delta \subset S^1 \times \R^2$ defined in \eqref{Ndelta/eqn}, where $\theta$ is the angular coordinate in $S^1$ and $(r,\varphi)$ are the polar coordinates in $\R^2$.
We fix such a contactomorphism $f\: (N_\delta, \xi_0) \to N$. Since $\Bd W$ is strongly pseudoconcave, the diffeomorphism $f$ is orientation-reversing.

For any positive number $\epsilon $ smaller than $\delta $, we take the totally real annulus 
\[A_{\epsilon}:=S^1\times \mathopen] -\epsilon, \epsilon \mathclose[ \subset \C \times \C,\]
and define an embedding $g_n \: A_{\epsilon }\to N_\delta$ by 
\[g_n(z, x)=(z, x z^n),\]
where $n$ is any integer which will give the attaching framing of the handle $H_a$, and where we consider $N_\delta \subset S^1 \times \C \subset \C^2$ up to the obvious identification $\R^2 \cong \C$.

Next, the composition $f\circ g_n \: A_\epsilon \to N$ admits a $C^{\infty }$-approximation to a real analytic totally real embedding $f_n \: A_{\epsilon}\to N$. Moreover, the curve $\widetilde L = f_n(S^1 \times \{0\}) \subset N$ is a real analytic transverse knot $C^\infty$-close to $L$.

The map $f_n$ is a real analytic diffeomorphism between totally real submanifolds 
$A_{\epsilon }\subset \C^2$ and $f_n(A_{\epsilon}) \subset W$, and hence, by Lemma 5.40 in Cieliebak and Eliashberg \cite{CE12},
it can be extended uniquely to a biholomorphism $\widetilde f_n \: U \to V$ between sufficiently small neighborhoods $U$ and $V$ of $A_{\epsilon }\subset \C^2$ and of $f_n(A_{\epsilon})\subset W'$, respectively, where $W'$ is a larger complex surface without boundary that contains $W$ as a complex domain.

Now, observe that for $a$ large enough, the attaching region of the standard concave handle $H_a$ is contained in $U$.
Then, we can attach $H_a$ holomorphically to $W$ along
$\widetilde L \subset \Bd W$ by means of the biholomorphism $\widetilde f_n$, obtaining a complex surface $\widetilde W = W\cup_{\widetilde f_n} H_a$. By construction, the topological framing of this 2-handle is represented by the knot $\widetilde L_n' = \widetilde f_n\left(S^1 \times \left\{\frac{\epsilon}{2}\right\}\right) \subset \Bd W$, which is a parallel copy of $\widetilde L$ in $\Bd W$, as it can be easily derived by considering the identification of $H_a$ with the topological 2-handle given by the diffeomorphism $h_a$ in equation \eqref{h-diffeo/eqn}. Moreover, $\widetilde L_n'$ is isotopic, inside the tubular neighborhood $N$ of $\widetilde L \subset \Bd W$, to $\widetilde L_0'$ minus $n$ full twists (because $f$ is orientation-reversing). Since $n \in \Z$ is arbitrary, any topological framing can be thus attained.

Finally, the general case immediately follows because $\Bd W$ can be pushed a little bit in $W$ near $L$ so that there it becomes real analytic, and then the above construction applies. It is thus enough to observe that this holomorphic handle addition can be considered along the original $L$, and we still indicate the resulting complex surface by $\widetilde W$.
\end{proof}

We note that the complex surface $\widetilde W$ has corners along a 2-torus in the boundary, corresponding to the intersection $\Bd W \cap \Bd H_a$. 
Next, we remove the corners and make the boundary strongly pseudoconcave. 

\begin{proposition}\label{concave boundary}
Let $W$, $\xi$ and $L$ be as in Proposition \ref{attach}.
Then, the result $\widetilde W$ of a holomorphic attaching of $H_a$ to $W$ along $L$ as above, contains a closed collar neighborhood $C$ of its boundary 
such that $\widetilde W' = \widetilde W - \Int C$ is a compact complex surface with smooth strongly pseudoconcave boundary, for any sufficiently large $a>0$. 
If the transverse knot $L$ is the positive push-off of a Legendrian knot $K \subset \Bd W$, 
then the contact structure on $\Bd \widetilde W'$ can be obtained from $\xi $ by a Legendrian surgery on $K$. 
By suitably choosing the attaching framing of the handle $H_a$, 
both contact $(\pm 1)$-surgeries on $K$ can be attained. 
\end{proposition}

\begin{proof}
Since the boundary $\Bd W$ is strongly pseudoconcave, 
there exists a strictly plurisuperharmonic function $\phi$ 
on a tubular neighborhood $U$ of $L$ in $W$ such that 
$\phi ^{-1}(0) \subset \Bd W$ is a solid torus around $L$ and $\phi (U)=\mathopen]-c,0\mathclose]$, for some $c>0$. We can assume that the attaching region of $H_a$ is mapped inside $U$ by the attaching biholomorphism, for any sufficiently large $a >0$.
For $c$ small enough, the core disk $D$ of $H_a$ is transversal to the level sets of $\phi$. 
Hence, we can take $a>0$ large enough so that the level sets of $\psi_a$ are transversal to the level sets of $\phi$, where $\psi_a$ is the function defined in Proposition \ref{psi/thm}.

Now, we take domains $U'=U-H_a(c)$ and $V=H_a(0,c)$ in $H_a$, 
and define a function $\lambda$ by 
\begin{align*}
{\lambda } =
\begin{cases}
\min(\phi, \psi_a) & \text{on } U'\cap V, \\
\phi & \text{on } U'-V, \\
\psi_a  & \text{on } V-U'. 
\end{cases}
\end{align*}
Then, $\lambda$ is continuous and strictly plurisuperharmonic. 
It is indeed continuous since  
\begin{align*}
\min(\phi , \psi )=
\begin{cases}
\psi_a  & \text{ on } \phi^{-1}(0)\cap V, \\
\phi  & \text{ on } \psi_a^{-1}(0)\cap U'.  
\end{cases}
\end{align*}
By Richberg's Theorem \cite{Ri}, it can be approximated from below 
by a smooth strictly plurisuperharmonic function $\widetilde\lambda$ 
which coincides with $\lambda$ away from a neighborhood of the set $\left\{\phi =\psi_a \right\}$. 
Since the holomorphic disk $D$ is transversal to the boundary $\Bd W$, 
there exists a vector field $X$ such that $X\phi >0$ and $X\psi_a >0$. 
By Corollary 3.20 in \cite{CE12}, 
$\widetilde{\lambda}$ can be arranged to satisfy $X\widetilde{\lambda}>0$. 
Hence, we may assume that $\widetilde{\lambda}$ has no critical points. 
Then, for any $d \in \mathopen]0, c \mathclose[$, we may take  $C = \widetilde{\lambda }^{-1}[-d, 0]$ and so
\[\widetilde W'=\widetilde W-\Int C\] 
is a complex surface with strongly pseudoconcave smooth boundary 
$\Bd \widetilde W'=\widetilde{\lambda }^{-1}(-d)$. 

Now, let $\eta $ be the induced contact structure on the concave boundary $\Bd \widetilde W'$, and let the transverse knot $L$ be a positive push-off of a Legendrian knot $K$. 
Then the contact manifold $(\Bd \widetilde W', \eta )$ is obtained from $(\Bd W, \xi )$ 
by some Legendrian contact surgery on $K$. This is proved by the following argument. 
First we take a standard neighborhood $N_K$ of $K$ 
that is a contact solid torus bounded by a convex torus with two parallel dividing curves. 
Since it contains the positive push-off $L$, 
we can also take the neighborhood $N$ of $L$ as in the proof of Proposition \ref{attach} inside $N_K$. 
Hence, the holomorphic handle attaching does not change the contact structure outside $N_K$. 
Moreover, the contact solid torus $\Bd \widetilde W'\cap H_a$ is tight since it is a hypersurface in $H_a\subset \C^2$. 
Therefore, $(\Bd \widetilde W', \eta )$ is the result of a Legendrian surgery on $K$. 

Recall now that the integer $n$ of the embedding $g_n$ in the proof of Proposition \ref{attach}, 
determines the attaching framing of the $2$-handle $H_a$. 
Since the framing can be arbitrary, we can realize both contact $(\pm 1)$-surgeries on $K$. 
\end{proof}

By putting $Y=\widetilde W'$, Theorem \ref{cobordism} immediately follows from Proposition \ref{concave boundary}.

\begin{remark}
Transversality of $L$ to $\xi $ is essential for taking an arbitrary attaching framing and also for modifying the original strictly plurisuperharmonic function. 
\end{remark}

\begin{remark}
The holomorphic handle attaching to the concave boundary cannot be achieved by a Weinstein handle. 
In the proof of Proposition \ref{attach}, the gluing domain is a small neighborhood of $A_{\epsilon }$. 
Consequently, the attaching $2$-handle must be thin enough. 
Hence, a Weinstein handle is useless for holomorphic handle attaching to the concave boundary 
since a thin Weinstein handle is a convex handle. 
\end{remark}

\section{Complex cobordisms and concave holomorphic fillings}\label{proofs}

In this section, we prove Theorem \ref{main2} and Theorem \ref{main}. 
The proof is a simple combination of our holomorphic handle attaching method and Ding-Geiges' result. 
The next proposition is a direct corollary to Theorem \ref{cobordism}. 

\begin{proposition}\label{concave filling}
Let $(M, \xi)$ be a closed positive contact $3$-manifold and let $K$ be a Legendrian knot in $(M, \xi)$. 
If $(M, \xi)$ admits a concave holomorphic filling, then the same holds for the contact manifolds $({M_K}^{\pm}, {\xi_K}^{\pm})$ obtained from $(M, \xi)$ by $(\pm1)$-contact surgery along $K$. 
\end{proposition}

\begin{lemma}\label{Kähler/thm}
Let $W$ be a Kähler complex surface with strongly pseudoconcave boundary, 
and let $\widetilde W$ be obtained from $W$ by attaching a holomorphic 2-handle $H_a$. 
Then, if $a$ is large enough, the Kähler form of $W$ can be extended over $H_a$.
\end{lemma}

\begin{proof}
Let $W'\supset W$ be a complex surface without boundary that contains $W$.
The attaching circle of $H_a$ has a Stein tubular neighborhood $U$ in $W'$, hence the Kähler form $\omega$ there admits a Kähler potential $h \: U \to \R$, that is $\omega = i\, \partial\bar\partial h$ in $U$. Clearly, $h$ can be extended a little bit inside the handle $H_a$ as a strictly plurisubharmonic smooth function,  if $a$ is large enough so that the attaching region of $H_a$ is contained in $U$.

In the open subset $V$ of $H_a$ where $h$ is already defined, let $k > 0$ be a constant and take the function $h_1(z_1,z_2)=h(z_1,z_2)+ k\log |z_1|$, which is strictly plurisubharmonic and satisfies $\partial\bar\partial h_1 = \omega$ because $\log |z_1|$ is harmonic. If $k>0$ is large enough then (up to taking a larger value for $a$) the function $h_1$ is strictly increasing with respect to $|z_1|$, in a sufficiently thin region $1-2\epsilon\leq |z_1|\leq 1$, for some $\epsilon >0$. Up to rescaling, we may further assume that $h_1(z_1,z_2)\leq 0$ for $|z_1|\leq 1-2\epsilon$ and $h_1(z_1,z_2)\geq 1$ for $|z_1|\geq 1-\epsilon$.

Let $\lambda \: \R \to \R$ be a smooth function such that $\lambda(x) = 0$ for $x\leq 0$, $\lambda$ is strictly convex (hence strictly increasing) in $\mathopen]0, 1\mathclose[$, and $\lambda =\id$ in $[1, {+\infty})$. Then, $\lambda\circ h_1$ vanishes for $|z_1| \leq 1-2\epsilon$, is strictly plurisubharmonic for $|z_1| > 1-2\epsilon$, and coincides with $h_1$ for $|z_1| \geq 1-\epsilon$. Hence, $\lambda\circ h_1$ may be extended by zero in the region $|z_1| \leq 1-2\epsilon$, yielding a function $h_2\: H_a \to \R$. 

Next let $h_3 \: H_a \to \R$ be the function defined by 
$h_3(z_1, z_2) = b(|z_1|^2 + |z_2|^2)$, where $b>0$ is a constant that may be chosen sufficiently small 
so that the inequality $h_2>h_3$ holds for $|z_1|\geq 1-\epsilon$.
 
Now, we take $h_4 = \max(h_2, h_3)$, which is continuous, strictly plurisubharmonic and coincides with $h_1$ for $|z_1|\geq 1-\epsilon$.
Finally, let $h_5$ be a smooth strictly plurisubharmonic approximation of $h_4$. 
Then, the Kähler form $\omega$ can be extended over $H_a$ by $i\, \partial\bar\partial h_5$.
\end{proof}

We are ready to prove Theorem \ref{main2} and Theorem \ref{main}. 

\begin{proof}[Proof of Theorem \ref{main2}] 
By removing a strongly pseudoconvex $4$-ball from an arbitrary compact complex surface $S$, 
we obtain a concave holomorphic filling of the standard contact structure on $S^3$. 
Hence, it follows from Theorem \ref{DG} and Proposition \ref{concave filling} that 
any closed positive contact $3$-manifold $(M, \xi )$ admits a concave holomorphic filling. 
Corresponding to infinitely many choices of $S$, 
there are infinitely many different concave holomorphic fillings of $(M, \xi )$. The last statement follows from Lemma \ref{Kähler/thm} in the Kähler case, with $S$ a Kähler surface, and otherwise by starting from a non-Kähler surface $S$ (for example a Hopf surface) in the non-Kähler case. 
\end{proof}

\begin{proof}[Proof of Theorem \ref{main}] 
We start with a Kähler concave holomorphic filling of $(M_2, \xi _2)$. 
By Theorem \ref{DG}, $(M_1, \xi_1)$ can be obtained from $(M_2, \xi_2)$ by a sequence of $(\pm 1)$-surgeries. 
Applying Theorem \ref{cobordism} repeatedly, the sequence can be realized as a complex cobordism. 
Namely, there exists a convex complex cobordism from $(M_1, \xi_1)$ to $(M_2, \xi _2)$.
\end{proof}

\begin{remark}
Even if the complex cobordism obtained by Theorem \ref{main} is Kähler, it is not necessarily a symplectic cobordism. 
For example, let $(M_1, \xi _1)$ be strongly symplectically fillable and $(M_2, \xi _2)$ be overtwisted.  
Then, there is no symplectic cobordism from $(M_1, \xi _1)$ to $(M_2, \xi _2)$ 
since an overtwisted contact $3$-manifold admits no weak symplectic filling. 
Therefore, on a Kähler convex complex cobordism $V$ from $(M_1, \xi _1)$ to $(M_2, \xi _2)$, 
there exists a symplectic form $\omega $ compatible with the complex structure, 
but the concave boundary $N_1$ never be $\omega $-concave, 
namely, there is no Liouville vector field near $N_1$ with respect to any compatible symplectic form. 
\end{remark}

\section*{Acknowledgements}
The first author was supported by the JSPS KAKENHI Grant Number JP17K14193. 

The second author is a member of GNSAGA, Istituto Nazionale di Alta Matematica ``Francesco Severi'', Italy. Moreover, he was partially supported by:
\begin{enumerate}
\item the 2013 ERC Advanced Research Grant 340258 TADMICAMT, Lehrstuhl Mathematik VIII, University of Bayreuth;
\item the project FRA-BEORCHIA-18, n. j961c1800138001, Università degli Studi di Trieste.
\end{enumerate}

\end{document}